\documentclass[a4paper,10pt]{amsart}
\usepackage{amsmath,amsfonts,amssymb,amsthm, amsfonts, amscd, url}
\usepackage[utf8]{inputenc}
\usepackage[english]{babel}
\usepackage{url}

\usepackage[cmtip,arrow]{xy}

\usepackage{pb-diagram,pb-xy}

\swapnumbers 
\theoremstyle{plain}
\newtheorem{thm}{Theorem}[section]
\newtheorem{prop}[thm]{Proposition}
\newtheorem{lemma}[thm]{Lemma}

\theoremstyle{remark}
\theoremstyle{definition}
\newtheorem{rem}[thm]{Remark}

\newtheorem{remdef}[thm]{Remark-Definition}
\newtheorem{remsdefs}[thm]{Remarks-Definitions}

\newtheorem{notasdefis}[thm]{Notations\,-Definitions}

\usepackage[english]{babel}

\title[Some Riemannian properties of $\mathbf{SU_n}$]{
Some Riemannian properties of $\mathbf{SU_n}$\\ endowed with a bi-invariant metric}

\author{Donato Pertici \and Alberto Dolcetti}

\begin{document}

\parindent 0pt
\selectlanguage{english}

\maketitle

\vspace*{-0.2in}

\begin{center}
{\scriptsize Dipartimento di Matematica  e Informatica, Viale Morgagni 67/a, 50134 Firenze, ITALIA

\vspace*{0.07in}

donato.pertici@unifi.it,  \   http://orcid.org/0000-0003-4667-9568

\vspace*{-0.03in}

alberto.dolcetti@unifi.it, \  http://orcid.org/0000-0001-9791-8122

}

\end{center}


\vspace*{-0.1in}

\begin{abstract}
We study some properties of $SU_n$ endowed with the Frobenius metric   $\phi$, which is, up to a positive constant multiple, the unique bi-invariant Riemannian metric on $SU_n$. In particular we  express the distance between $P, Q \in SU_n$ in terms of eigenvalues of $P^*Q$; we compute the diameter of $(SU_n, \phi)$ and we  determine its diametral pairs; we prove that  the set of all minimizing geodesic segments with endpoints $P$, $Q$ can be parametrized by means of a compact connected submanifold of $\mathfrak{su}_n$,  diffeomorphic to a suitable
complex Grassmannian depending on $P$ and $Q$.

\end{abstract}


{\small \tableofcontents}

\renewcommand{\thefootnote}{\fnsymbol{footnote}}

\renewcommand{\thefootnote}{\arabic{footnote}}
\setcounter{footnote}{0}

\vspace*{-0.3in}

{\small {\scshape{Keywords.}} Special unitary group, Frobenius metric, bi-invariant metrics on a compact Lie group, minimizing geodesics, homogeneous space, complex Grassmannian, diameter, diametral pairs,
generalized principal logarithm; 
\smallskip

{\small {\scshape{Mathematics~Subject~Classification~(2020):}}  53C35, 15B30, 22E15.

{\small {\scshape{Grants:}}
supported by GNSAGA-INdAM and by MUR-PRIN project ``Multilinear Algebraic Geometry''.

\section*{Introduction}

Let $G$ be a closed subgroup of the unitary group $U_n$ and let $\mathfrak{g} \subseteq \mathfrak{u}_n$ be its Lie algebra. It is natural to consider on $G$ the \emph{Frobenius} (or \emph{Hilbert-Schmidt}) \emph{metric} $\phi$ induced on $G$ by the Euclidean metric of $\mathbb{C}^{n^2}$ (identified with $\mathfrak{gl}_n(\mathbb{C})$: the Lie algebra of complex square matrices of order $n$). Since $G \subseteq U_n$, the metric $\phi$ is bi-invariant on $G$. Furthermore, if $G$ is also absolutely simple and connected (e.g. the classical groups $SO_n$ for $n \ge 3$ and $n \ne 4$, $SU_n$ for $n \ge 2$ and $Sp_n$ for $n \ge 1$), every bi-invariant $(0, 2)$-tensor on $G$ (such as the Killing metric) is a constant
multiple of $\phi$ (see, for instance, \cite[Prop.\,1.8]{DoPe2023}). In particular, all bi-invariant $(0, 2)$-tensors on such groups are Riemannnian or anti-Riemannian metrics and have the same geodesics and the same isometries.

In \cite[Thm.\,2.3]{DoPe2023} we determined all isometries of any absolutely simple, compact, connected, real Lie group endowed with any bi-invariant metric; from that result, we deduced the list of all isometries of $SO_n$ ($n \ge 3$, $n \ne 4$), $SU_n$ ($n \ge 2$) and $Sp_n$ ($n \ge 1$), endowed with any bi-invariant metric (\cite[Thm.\,2.5\,(c),\,(d)]{DoPe2023}). Furthermore, by means of suitable different methods, we determined the list of all isometries also of $SO(4)$ and of $U(n)$ ($n \ge 2$), all endowed with the Frobenius metric (\cite[Thm.\,3.5, Thm.\,4.7]{DoPe2023}).

On the other hand, in \cite{PD2022}, we studied the Riemannian properties of many suitable subgroups $G$ of the unitary group $U_n$ (we called them  SVD-closed) endowed with the Frobenius metric $\phi$.
In particular, we expressed   the distance between points and we parametrized the set of all minimizing geodesic segments of $(G, \phi)$ with arbitrary endpoints $P, Q$,
by means of a set, denoted by $\mathfrak{g}$--$plog (P^*Q)$, consisting of generalized principal logarithms of $P^*Q$  belonging to $\mathfrak{g}$; we proved that this last set is non-empty, it is a union of finitely many compact submanifolds  of the Lie algebra $\mathfrak{u}_n$ and that each of these submanifolds is diffeomorphic to a suitable homogeneous space (\cite[Thm.\,5.7 and Thm.\,6.5]{PD2022}). For a lot of families of SVD-closed subgroups of $U_n$ (among them $U_n$, $SO_n$, $Sp_n$ and others) we also determined the (Frobenius) diameter of each subgroup and the structure of each connected component of $\mathfrak{g}$--$plog (P^*Q)$ as a homogeneous space (\cite[Prop.\,6.7 and \S\,7, \S\,8]{PD2022}).

Most classical subgroups of $U_n$ are SVD-closed with the exception of $SU_n$, when $n \ge 3$.
For this reason, in this paper we try to obtain similar results for $(SU_n, \phi)$.
In particular, we are able to  express the distance between  points, to compute the diameter of $(SU_n, \phi)$, to  determine its diametral pairs (i.e. the pairs of points whose distance is equal to the diameter) and to  parametrize the set of all minimizing geodesic segments with endpoints $P$, $Q$\,, 
by means of a compact connected submanifold of $\mathfrak{su}_n$, which in general is not equal to $\mathfrak{su}_n$\!--\,$plog (P^*Q)$), but it is always diffeomorphic to a suitable
complex Grassmannian depending on $P$ and $Q$ (Theorem \ref{distanza} and Theorem \ref{diametral-points}).
Since we already studied generalized principal logarithms of some different types of matrices in \cite{DoPe2018}, \cite{PD2022}, \cite{Pe2023}, we also determine the conditions on $Q \in SU_n$, so that the set $\mathfrak{su}_n$\!--\,$plog (Q)$ is  non-empty (Proposition \ref{sunlog}).

The euclidean metric on $\mathfrak{gl}_n(\mathbb{C})$ can be also restricted to the Lie group $GL_n(\mathbb{C})$ of invertible complex matrices; this restriction is a Riemannian metric on $GL_n(\mathbb{C})$, but it is not bi-invariant. On the other hand $GL_n(\mathbb{C})$ can be endowed, in a natural way, with a semi-Riemannian metric, which, on the contrary,  is bi-invariant: the so-called \emph{trace metric} $g$,
 defined, at the identity matrix, by $g(X,Y) := Re(Tr(XY))$  (the real part of the trace of the product $XY$) for every $X, Y \in \mathfrak{gl}_n(\mathbb{C})$. We remark that, for any closed subgroup $G$ of $U_n$, the restriction to $G$ of the trace metric $g$ agrees with the anti-Riemannian metric $-\phi$ (the opposite of the Frobenius metric of $G$) and hence it determines on $G$ the same Riemannian structure of $\phi$.

The restrictions of the trace metric define  Riemannian structures on the space of positive definite hermitian matrices and on the space of positive definite symmetric real matrices. Such manifolds have a remarkable interest in many frameworks of both pure and applied mathematics (see for instance, among many others,
\cite{BridHaef1999}, \cite{Mo2005}, \cite{Bha2007},  \cite{bar2008}, 
\cite{MoZ2011},

 \cite{NieBha2013},
 \cite{NieBar2019},
   \cite{DoPe2019},
   
 \cite{DoPe2021},
 \cite{CBG2021},
 \cite{Bou2023},
 
   \cite{CMB2023}, 
      \cite{Nie2023a}, \cite{Nie2023b},
      
    \cite{ZZS2023}, \cite{Nieu2024}). 
We have also studied the restrictions of the trace metric on other submanifolds of $GL_n(\mathbb{C})$, where they induce semi-Riemannian (but not Riemannian) structures 
(see \cite{DoPe2015},  \cite{DoPe2019},

 \cite{DoPe2020}).

The present paper is organized as follows.

In Section 1, we recall the main notations and the basic facts used in this paper. We introduce the Frobenius metric $\phi$ on any closed connected subgroup $G$ of $U_n$, having $\mathfrak{g} \subseteq \mathfrak{u}_n$ as Lie algebra. In particular, we observe that (Proposition \ref{distanza-punti}):

(i) for every $P, Q \in G$, the (Frobenius) distance $d(P,Q)$ is equal to the minimum of the (Frobenius) norm $\Vert X \Vert_{_\phi}$, where $X \in \mathfrak{g}$ is a logarithms of $P^*Q$; 

(ii) the map: $X \mapsto \gamma(t):= P \exp(t X)\ \ (0 \leq t \leq 1)$ \ is a bijection from  the set 

$\{ X \in \mathfrak{g} \, : \,  \exp(X) = P^*Q \mbox{ and } \Vert X \Vert_\phi = d(P,Q) \}$
onto the set of minimizing geodesic segments of \ $(G, \phi)$ with endpoints $P$ and $Q$. 

Next Sections are devoted to study some Riemannian properties of the group $SU_n$ endowed with the Frobenius metric $\phi$.

Aim of Section 2 is to find an explicit  expression of the function

 $m(Q):=\min\{ \Vert X \Vert_{_\phi}^2 \ : \ X \in \mathfrak{su}_n \mbox{ and } \exp(X)=Q \} = d(I_n, Q)^2$ for any  matrix $Q \in SU_n$ ($I_n$ is the identity matrix); the value of $m(Q)$ is obtained in terms of the eigenvalues of $Q$ (Proposition \ref{m-valutation}).

In Section 3, for every $Q \in SU_n$, we analyse the set 

$\Theta(Q):= \{X\in \mathfrak{su}_n : \exp(X)=Q\,, \  \Vert X \Vert_{_\phi}^2= m(Q) =d(I_n, Q)^2\}$, which (as stated above) parametrizes the set of minimizing geodesic segments of \ $(SU_n, \phi)$ with endpoints $I_n$ and $Q$.
 We prove that $\Theta(Q)$ is a compact connected submanifold of $\mathfrak{su}_n$, diffeomorphic to a suitable complex Grassmannian (Proposition \ref{Theta(Q)}). Furthermore we also determine the geometric structure of
$\mathfrak{su}_n$\!--\,$plog (Q)$  (Proposition
\ref{sunlog}).

In Section 4, we obtain some fundamental results concerning the diameter of $(SU_n, \phi)$ and its diametral pairs  (Proposition \ref{diameter}).

In the last Section 5, by means of the Propositions proved in Sections 2, 3, 4, we collect and prove the main results of this paper: the already mentioned Theorems \ref{distanza} and \ref{diametral-points}.

\section{Preliminary facts}

\begin{notasdefis}\label{notazioni}
a) We denote by $\mathbb{C}$ the field of complex numbers, by $\textbf{i}$ its imaginary unit,  by $\mathbb{R}$ the field of real numbers and by $\mathbb{Z}$ the ring of integers. The integer part of any $x \in \mathbb{R}$ is denoted by $\lfloor x \rfloor$.
We denote by $|w|$ \,and \,$e^w := \sum\limits_{i=0}^{+ \infty} \dfrac{w^i}{i!}\,,$  respectively, the modulus and the exponential of an arbitrary complex number $w$, while, for any $z \in \mathbb{C} \setminus \{0\}$, we denote by $\arg(z) \in (- \pi , \pi]$ the \emph{principal value of the argument} of $z$ and by  $\log(z):= \ln|z|+ \arg(z) \textbf{i}$ the  \emph{principal value of the logarithm} of $z$. Clearly $\log(z)$ is the unique complex logarithm of $z$, whose imaginary part lies in the interval $(- \pi , \pi]$. 
Note that, when $|z|=1$, we have  \ $\log(z)=\arg(z) \textbf{i}$, so, in particular, $\log(-1)=\pi \textbf{i}\,.$ Of course, if $|z|=|w|=1$, we have $z=w$ if and only if $\arg(z)=\arg(w)$ \,(i.e. $\log(z)=\log(w)$).

\smallskip

b) We denote by \ $GL_n(\mathbb{C})$ \ (with $n\ge 1$) \ the Lie group of invertible complex square matrices of order $n$, by
\ $\mathfrak{gl}_n (\mathbb{C})$ its Lie algebra consisting of all $n\times n$ complex matrices and by $I_{_n}$ the identity matrix of $GL_n(\mathbb{C})$. 
For every $A \in \mathfrak{gl}_n(\mathbb{C})$, \ $A^T$, $\overline{A}$, $A^{*} := \overline{A}^T$ and $A^{-1}$ (provided that $A$ is invertible) are respectively transpose, conjugate, adjoint and inverse of the matrix $A$, tr$(A)$ is its trace, $\det(A)$ is its determinant, while $\exp(A):= \sum\limits_{i=0}^{+ \infty} \dfrac{A^i}{i!}\in GL_n(\mathbb{C})$ denotes the exponential of $A$. It is clear that, if $\lambda_1, \cdots , \lambda_n$ are the $n$ eigenvalues of $A \in \mathfrak{gl}_n(\mathbb{C})$, then $e^{\lambda_1}, \cdots , e^{\lambda_n}$ are the $n$ eigenvalues of $\exp(A)$. For convenience, we agree that a complex number $\eta$ is an \emph{eigenvalue of multiplicity} $0$ of a matrix $A \in \mathfrak{gl}_n(\mathbb{C})$ if $\eta$ is not a usual eigenvalue of $A$, i.e. if $\eta$ is not a root of the characteristic polynomial of $A$.

For any integer $n \geq 1$, we denote by
\ $U_n:=\{A \in \mathfrak{gl}_n (\mathbb{C}): AA^*=I_{_n}\}$ \ the unitary group of degree $n$ and by \ $\mathfrak{u}_n:= \{X \in\mathfrak{gl}_n (\mathbb{C}): X=-X^*\}$ \ its Lie algebra of skew-hermitian matrices; we also denote by \ $SU_n:=\{A \in U_n: \det(A)=1\}$ \ the special unitary group of degree $n$ and by \ $\mathfrak{su}_n:= \{X \in\mathfrak{u}_n : \mbox{tr}(X)=0\}$ \ its Lie algebra. Finally, for any $U \in U_n$\,, we denote by $Ad_{_U}: \mathfrak{gl}_n (\mathbb{C}) \to \mathfrak{gl}_n (\mathbb{C})$ the map defined by $Ad_{_U}(X):=UXU^*$ (for every $X \in \mathfrak{gl}_n (\mathbb{C})$) and we still denote by $Ad_{_U}$ the restriction of this map to any subset of $\mathfrak{gl}_n (\mathbb{C})\,;$ so we can write both $Ad_{_U}: SU_n \to SU_n$ and $Ad_{_U}: \mathfrak{su}_n \to \mathfrak{su}_n\,.$ Note that $Ad_{_U}$ commutes with the exponential map.

\smallskip

c) Let $G$ be a closed connected subgroup of $GL_n(\mathbb{C})$ having $\mathfrak{g}\subset \mathfrak{gl}_n (\mathbb{C})$ as its Lie algebra. If $B \in G, \ A \in \mathfrak{g}$ and $\exp(A)=B$, we say that  $A$ is a $\mathfrak{g}$--\emph{logarithm} of $B$. Moreover, if $A$ is a $\mathfrak{g}$--logarithm of $B$ such that \ $- \pi \le Im(\lambda) \le\pi$ \ for every eigenvalue $\lambda$ of $A$, we say that $A$ is a  \emph{generalized principal} $\mathfrak{g}$--\!\!\!\emph{ logarithm} of $B$. 

We denote by \ $\mathfrak{g}$--$plog(B)$ \ the set of all generalized principal $\mathfrak{g}$--logarithms of $B \in G$.

Recall that, if $G$ is also compact, the exponential map $\exp: \mathfrak{g} \to G$ is surjective (see, for instance,  \cite[Thm.\,5.12, p.\,102]{Sepa2007}).

\smallskip

d) If $B_1, \cdots , B_t$ are square matrices (of various orders), then $\bigoplus\limits_{j=1}^t B_j=B_1 \oplus \cdots \oplus B_t$ denotes the block diagonal square matrix with $B_1, \cdots , B_t$ on its diagonal. 

Clearly we have \ $\exp(\bigoplus\limits_{j=1}^t B_j)= \bigoplus\limits_{j=1}^t \exp(B_j)$ , for every $ B_1, \cdots , B_t\,$.

If $\mathcal{S}_1, \dots , \mathcal{S}_t$ are sets of square matrices, then  $\bigoplus\limits_{j=1}^t \mathcal{S}_j=\mathcal{S}_1 \oplus \dots \oplus \mathcal{S}_t$ denotes the set of all matrices $\bigoplus\limits_{j=1}^t B_j\,,$ with $B_j \in \mathcal{S}_j$\ \, for every $j$.
 \end{notasdefis} 

\begin{lemma}\label{commutazione}
Let $n_1,\cdots, n_t$ be positive integers such that $n=\sum\limits_{j=1}^t n_j$ and let $\lambda_1,\cdots, \lambda_t$ be distinct complex numbers. Set $A:=\bigoplus\limits_{j=1}^t \lambda_j I_{_{n_j}}\,.$ Then a matrix $B\in U_n$ commutes with $A$ if and only if \ $B\in \bigoplus\limits_{j=1}^t U_{n_j}\,.$
\end{lemma}

For a proof of the Lemma, see for instance \cite[Lemma\,1.3]{Pe2023}.

\begin{remdef}\label{frob}
We denote by ${\phi}$  the \emph{Frobenius} (or \emph{Hilbert-Schmidt}) \emph{positive definite  real scalar product} on $\mathfrak{gl}_n (\mathbb{C})$ defined by ${\phi}(A, B) := \mbox{Re}(\mbox{tr}(A B ^*))$, and we denote by $\Vert A \Vert_{_\phi} := \sqrt{{\phi}(A, A)} = \sqrt{\mbox{tr}(AA^*)}$ the related \emph{Frobenius norm}. Note that, if $A \in \mathfrak{u}_n$, then 
$\Vert A \Vert_{_\phi}^2 =  -\mbox{tr}(A^2)$. Since the eigenvalues of the skew-hermitian matrix $A$ are purely imaginary, we also get $\Vert A \Vert_{_\phi} = \sqrt{-\mbox{tr}(A^2)} = \sqrt{\sum\limits_{j=1}^n |\lambda_{_j}|^2}$, where $\lambda_{_1} , \cdots , \lambda_{_n}$ are the $n$ (possibly repeated) eigenvalues of $A$.
For any arbitrary closed subgroup $G$ of $U_n$, we still denote by $\phi$ the Riemannian metric on $G$, obtained by restriction of the Frobenius scalar product of $\mathfrak{gl}_n (\mathbb{C})$. It is easy to check that the metric $\phi$ (called the \emph{Frobenius metric} of $G$) is bi-invariant on $G$ and we have $\phi_{_A}(X, Y) = -\mbox{tr}(A^* X A^* Y)$, for any pair $X, Y$ of tangent vectors to $G$ at $A$ and for any $A \in G$. We denote by $d:=d_{_{(G,\phi)}}$ the distance on $G$ induced by $\phi$ and by $\delta (G, \phi) : = \sup \{d(P,Q) \, : \, P, Q \in G \}$ the \emph{diameter} of $G$ with respect to $d$. Of course $\delta (G, \phi) = d(A, B) < +\infty$, for some $A, B \in G$, because $G$ is compact.

 If $A,B \in G$ and $d(A, B)=\delta (G, \phi)$, we say that $A$ and $B$ are \emph{diametral in} $(G, \phi)$ or, equivalently, that they form a \emph{diametral pair of} $(G, \phi)$. Since $\phi$ is bi-invariant on the compact Lie group $G\subset U_n\,,$ \,it is clear that every point of $G$ belongs at least to a diametral pair of $(G, \phi)$. For some results on  diametral pairs in real orthogonal groups, see \cite[Prop.\,4.18]{DoPe2018}.
\end{remdef}

\begin{prop}\label{Levi-Civita_geod_Un} Let $G$ be a closed subgroup of $U_n$ and let $\mathfrak{g} \subseteq \mathfrak{u}_n$ be its Lie algebra. Then $(G, \phi)$ is a globally symmetric Riemannian manifold with non-negative sectional curvature, whose Levi-Civita connection agrees with the $0$-connection of Cartan-Schouten of $G$. The geodesics of $(G, \phi)$ are the curves \ 
$t \mapsto P\exp({t X})\ (\mbox{with}\ t \in \mathbb{R})$,\ 
for every $X \in \mathfrak{g}$ and $P \in G$; \ furthermore $(G, \phi)$ is a totally geodesic submanifold of $(U_n, \phi)$.
\end{prop}

For a proof of Proposition \ref{Levi-Civita_geod_Un}, we refer, for instance, to \cite[\S\,2.2]{AlBet2015}.

\begin{prop}\label{distanza-punti}
Let $G$ be a closed connected subgroup of $U_n$ and let $\mathfrak{g} \subseteq \mathfrak{u}_n$ be its Lie algebra. 
Then, for every $P, Q \in G$, the (Frobenius) distance $d(P, Q)$ is equal to the minimum of the set
\ $\{\ \Vert X \Vert_{_\phi} \ : \ X \mbox{\ is  a} \ \mathfrak{g}\mbox{--logarithm of}\   P^* Q\ \}$.
 
Furthermore, the map: $X \mapsto \gamma(t):= P \exp(t X)\ \ (0 \leq t \leq 1)$ \ is a bijection from  the set $\{X \, : \, X \mbox{ is a }\mathfrak{g}\mbox{--logarithm of }    P^* Q\, \mbox{ with } \Vert X \Vert_{_\phi}=d(P,Q) \}$ onto the set of minimizing geodesic segments of \ $(G, \phi)$ with endpoints $P$ and $Q$.
\end{prop}

\begin{proof}
Any geodesic segment $\gamma$, joining $P$ and $Q$, can be parametrized by $\gamma(t) = P \exp({t X})$ ($t \in[0,1]$), with $X \in \mathfrak{g},\  \exp(X) =  P^*Q$, and its length is $\sqrt{-tr(X^2)}= \Vert X \Vert_{_\phi}$; 
so we conclude by means of the classical Hopf-Rinow theorem.
\end{proof}

\begin{rem} Let $G$ be as in Proposition \ref{distanza-punti}.
We resume some results obtained in \cite{PD2022}.

In Thm. 6.5, we proved that, if $G$ is  SVD-closed too (for definition see \S\,4 of that paper) and $P, Q \in G$, then

a) $
d(P, Q)= \sqrt{\sum\limits_{j=1}^n |\log(\mu_{_j})|^2}
$, \ \  where  \ $\mu_{_1}, \cdots , \mu_n$ \  are the $n$ eigenvalues of $P^*Q$;

b)  $\mathfrak{g}$--$plog (P^*Q) \ = \ \{X \, : \, X \mbox{ is a }\mathfrak{g}\mbox{--logarithm of }    P^* Q\, \mbox{ with } \Vert X \Vert_{_\phi}=d(P,Q) \}$ and so 

$\mathfrak{g}$--$plog (P^*Q)$ parametrizes the set of minimizing geodesic segments of \ $(G, \phi)$ with endpoints $P$ and $Q$.

In Prop.\,6.7 and Rem.\,6.8, we computed the diameter of suitable families of SVD-closed subgroups of $U_n$, among them all classical subgroups of $U_n$ with the exception of $SU_n$.

For the same families of SVD-closed subgroups $G$ of $U_n$, in Sections 7 and 8, we proved that, for any matrix $M \in G$, the set $\mathfrak{g}$--$plog (M)$  is a disjoint union of finitely many simply connected compact submanifolds of $\mathfrak{u}_n$, each of them is diffeomorphic to a symmetric (homogeneous) space, we explicitly determined.

As remarked in \cite[Rem.\,3.7]{PD2022}, $SU_n$ is not SVD-closed as soon as $n \ge 3$ and therefore the Riemannian manifold $(SU_n, \phi)$  was excluded from that study. Hence, the goal of the present paper is to research for similar results in case of $(SU_n, \phi)$.
\end{rem}

\section{About the distance in $(SU_n, \phi)$ of any matrix from $I_n$}

\begin{remsdefs}\label{prelim-Q-q}
a) Fix a matrix  $Q \in SU_n$ ($n \ge 2$), and let $\mu_1 \cdots , \mu_n$ be the (possibly repeated) eigenvalues of $Q$, so we have $\mu_1 \mu_2 \cdots \mu_n =1 $; \ it is also known that  \\$| \mu_1 | = | \mu_2 | = \cdots | \mu_n | = 1,$ 
\ so $\log(\mu_j) = \arg(\mu_j) \textbf{i}$, for $j = 1, \cdots , n$.
Up to reordering the eigenvalues, we can assume \ \ \ 
$- \pi < \arg(\mu_1) \le \arg(\mu_2) \le \cdots \le \arg(\mu_{n}) \le\pi .$
 
Since\ $1= \prod\limits_{j=1}^n e^{\log(\mu_j)}=e^{(\arg(\mu_1)+\cdots+\arg(\mu_n))\textbf{i}} ,$ we also have $\sum\limits_{j=1}^n \arg(\mu_j) = 2k\pi$, with $k \in \mathbb{Z}.$

We also denote by $s=s(Q) \ge 0$ the multiplicity of $-1$ as an eigenvalue of $Q$.

When $s \ge 1$ we have
$\arg(\mu_j)= \pi$ \ if and only if \ \ $n-s+1\le j \le n$, while if $s =0$ we have $\arg(\mu_j) \ne \pi$ \  for every $j= 1, \cdots , n$. 

\smallskip

b) With the same notations as above, the eigenvalues of the inverse $Q^*$ of $Q$ \ are \ $\overline{\mu}_1, \cdots, \overline{\mu}_n$. Hence, if $s=0$  we have $\arg(\overline{\mu}_j)=-\arg(\mu_j)$ for every $j=1, \cdots, n$, while if $s\ge 1$ we have $\arg(\overline{\mu}_j)=-\arg(\mu_j)$ for every $j=1, \cdots, n-s$\ \ and\ $\arg(\overline{\mu}_j)=\arg(\mu_j)=\pi$ \ for every $j=n-s+1, \cdots, n$. In particular, in any case we have 
$s(Q^*)=s(Q)=s$ and
$-\pi< \arg(\overline{\mu}_{n-s})\le \cdots \le \arg(\overline{\mu}_1)<\pi$; moreover, if $s \ge 1$ we have $\arg(\overline{\mu}_j)= \pi$ \ for every $j=n-s+1, \cdots , n$.

\smallskip

c) We set $\zeta(Q):= \dfrac{1}{2 \pi} \sum\limits_{j=1}^n \arg(\mu_j) $. 

From  (a),  $\zeta(Q)$ is an integer dependent only on the principal arguments of the eigenvalues of $Q$, while, from (b), it is easy to check that we have \ $\zeta(Q^*)=s(Q)-\zeta(Q)$.

 From this last equality it is easy to deduce that 
$s(Q) -\lfloor \dfrac{n}{2} \rfloor \le \zeta(Q) \le \lfloor \dfrac{n}{2} \rfloor\ $ and also that the assumption \ $\zeta(Q) \ge \zeta(Q^*)$ \ implies \ $\zeta(Q)\ge 0\,.$
\end{remsdefs}

\begin{remdef}\label{admissibility}
Consider any integer $\zeta$ such that \ $-\lfloor \dfrac{n}{2} \rfloor \le \zeta \le \lfloor \dfrac{n}{2} \rfloor\ $ \ ($n\ge 2$); we say that an $n$-tuple ($\alpha_1, \cdots, \alpha_n$) of real numbers is $\zeta$-\emph{admissible} if 

$-\pi<\alpha_1\le \alpha_2\le \cdots \le \alpha_n \le \pi$ \ and \ $\sum\limits_{j=1}^n \alpha_j = 2 \pi \zeta\ .$

As seen in Remarks-Definitions \ref{prelim-Q-q} (a) and (c), the eigenvalues $\mu_1, \cdots , \mu_n$ of any matrix $Q$ of $SU_n$ can be reordered so that the $n$-tuple ($\arg(\mu_1), \cdots , \arg(\mu_n)$) is $\zeta(Q)$-admissible, being $\zeta(Q)$ the integer belonging to $\big[\!-\!\lfloor \dfrac{n}{2} \rfloor , \lfloor \dfrac{n}{2} \rfloor\,\big]$ defined in Remarks-Definitions \ref{prelim-Q-q} (c).
Conversely, it is clear that, if ($\alpha_1, \cdots, \alpha_n$) is any $\zeta$-admissible $n$-tuple of real numbers (with $\zeta \in \big[\!-\!\lfloor \dfrac{n}{2} \rfloor , \lfloor \dfrac{n}{2} \rfloor\,\big] \cap \mathbb{Z}$), there is at least one matrix $Q \in SU_n$ \ whose $n$ eigenvalues are  $\mu_1=e^{\alpha_{\!1}\!\textbf{i}}, \cdots,\mu_n= e^{\alpha_{\!n}\!\textbf{i}}\,;$ so we have $\alpha_1=\arg(\mu_1), \cdots, \alpha_n=\arg(\mu_n)$ \ and \ $\zeta=\zeta(Q)\,.$ 
\end{remdef}

\begin{remsdefs}\label{def-m(Q} a) Let $m: SU_n \to \mathbb{R}$ be the function defined by  

$m(Q):=\inf \{ \Vert X \Vert_{_\phi}^2 \ : \ X \in \mathfrak{su}_n \mbox{ and } \exp(X)=Q \}$

and, for any $Q \in SU_n$, let

$\Theta(Q):= \{X\in \mathfrak{su}_n : \exp(X)=Q\,, \  \Vert X \Vert_{_\phi}^2= m(Q)\}$.

From Proposition \ref{distanza-punti}, it follows that, for every $Q \in SU_n\,$, the map $X \mapsto \Vert X \Vert_{_\phi}^2$ has an absolute minimum on the set of $\mathfrak{su}_n\mbox{--logarithms of } Q\,;$\, so, for any $Q \in SU_n\,,$ we can write

$m(Q)=\min \{ \Vert X \Vert_{_\phi}^2 \ : \ X \in \mathfrak{su}_n \mbox{ and } \exp(X)=Q \} = d(I_n,Q)^2 $ (where $d$ is the distance induced by the Frobenius metric of $SU_n$), so that the set $\Theta(Q)$ is non-empty.

Now, fix $Q \in SU_n$ ($n \ge 2$) and denote $\mu_1 \cdots , \mu_n$ the $n$ eigenvalues of $Q$, so that the $n$-tuple ($\arg(\mu_1) \cdots , \arg(\mu_n)$) is $\zeta(Q)$-admissible, with 
$\zeta(Q) \in \big[\!-\!\lfloor \dfrac{n}{2} \rfloor , \lfloor \dfrac{n}{2} \rfloor\,\big]\cap \mathbb{Z} $. 
We have

(*) \ \ \ \ $m(Q) \ge \sum\limits_{j=1}^n \big(\!\arg(\mu_j)\big)^2$.

Indeed by definition we have:
$m(Q) \ge \min\{ \Vert X \Vert_{_\phi}^2 \ : \ X \in \mathfrak{u}_n \mbox{ and } \exp(X)=Q \}$ and this last is equal to $\sum\limits_{j=1}^n \big(\!\arg(\mu_j)\big)^2$ \ \ (by \cite[Rem.\,5.4 and Prop.\,5.5]{PD2022}).

Note that, if $X$ is a $\mathfrak{su}_n$--logarithm of $Q$, then the $n$ eigenvalues of $X$ are of the form 

$(\arg(\mu_1) +2 k_1 \pi) \textbf{i} , \cdots , (\arg(\mu_n)+ 2 k_n\pi) \textbf{i}$, for some integers $k_1, \cdots, k_n$ such that 

$\sum\limits_{j=1}^n (\arg (\mu_j) + 2 k_j \pi) = 2 \pi \, \big( \zeta(Q) + \sum\limits_{j=1}^n k_j \big) =0$, \ \,i.e. \ $\sum\limits_{j=1}^n k_j = - \zeta(Q)\,.$

Conversely, if $h_1, \cdots, h_n \in \mathbb{Z}$ and  $\sum\limits_{j=1}^n h_j = - \zeta(Q)$ \ (i.e.
$\sum\limits_{j=1}^n (\arg (\mu_j) + 2 h_j \pi) =0$), it is easy to determine a matrix $X_0 \in \mathfrak{su}_n$ such that $\exp(X_0)= Q$, whose $n$ eigenvalues are 

$(\arg(\mu_1) +2 h_1 \pi) \textbf{i} , \cdots , (\arg(\mu_n)+ 2 h_n\pi) \textbf{i}$, and therefore 
$\Vert X_0 \Vert_{_\phi}^2=\sum\limits_{j=1}^n \big(\!\arg (\mu_j) + 2 h_j \pi \big)^2$.
Hence we get
\ $m(Q) = \min \{  \sum\limits_{j=1}^n \big(\!\arg (\mu_j) + 2 k_j \pi \big)^2 \ : \ k_1, \cdots k_n \in \mathbb{Z}  \mbox{ and } \sum\limits_{j=1}^n k_j = - \zeta(Q)\}$.

Let
$W_Q:= \{ \underline{k}:=(k_1, \cdots , k_n) \in \mathbb{Z}^n \ : \ \sum\limits_{j=1}^n k_j = -\zeta(Q)\}$ and  denote by 
$\psi_{_Q}: W_Q \to \mathbb{R}$ the map defined by \ $\psi_{_Q}(\underline{k})=\psi_{_Q}(k_1, \cdots , k_n) = \sum\limits_{j=1}^n \big(\!\arg(\mu_j) + 2 k_j \pi\big)^2$; 
\ then we can write

$m(Q) = \min \{ \psi_{_Q}(\underline{k}) \ : \  \underline{k} \in W_Q\}$.

Next, it will be also useful to consider the map \ $\Delta_{_Q}: W_Q \to \mathbb{Z}$ defined by 

$\Delta_{_Q}(\underline{k}) := \max\{k_1, \cdots , k_n\} - \min\{k_1, \cdots , k_n\}$,  \ with $\underline{k}=(k_1, \cdots , k_n) \in W_Q$,\ \ and  the set \ 
${Z_Q} := \{ \underline{k}  \in W_Q : \Delta_{_Q}(\underline{k}) \le 1\}$. 

Note that $Z_Q\ne \emptyset$. Indeed, if $\zeta=\zeta(Q) \ge 0$, then $ (\,\underbrace{0, \cdots , 0}_{n-\zeta}, \, \underbrace{-1, \cdots, -1}_{\zeta}\,) \in Z_Q$, while if $\zeta=\zeta(Q) < 0$, then $ (\, \underbrace{1, \cdots , 1 }_{-\zeta}, \, \underbrace{0, \cdots, 0}_{n+\zeta} \, ) \in Z_Q$.

\smallskip

b) Note that the map $m : Q \mapsto m(Q)$ depends only on the $\zeta(Q)$-admissible $n$-tuple $(\arg(\mu_1), \cdots, \arg(\mu_n))$, so we can  unambiguously write  \ $m(\arg(\mu_1), \cdots, \arg(\mu_n)):=m(Q)\,.$ 
From this fact and from Remark-Definition \ref{admissibility}, in the sequel we will also be able to consider $m$ as a function only of an arbitrary $\zeta$-admissible $n$-tuple $(\alpha_1, \cdots, \alpha_n)$ of real numbers such that $\zeta \in \big[\!-\!\lfloor \dfrac{n}{2} \rfloor , \lfloor \dfrac{n}{2} \rfloor\,\big] \cap \mathbb{Z}\,.$

\smallskip

c) Keeping in mind Remarks \ref{prelim-Q-q} (b), it is easy to check that we have $m(Q^*)=m(Q)$, for any $Q \in SU_n$, or equivalently, \ $m(\alpha_1, \cdots, \alpha_{n-s}, \underbrace{\pi, \cdots, \pi}_{s})=m(-\alpha_{n-s}, \cdots, -\alpha_{1}, \underbrace{\pi, \cdots, \pi}_{s}),$ for any $\zeta$-admissible $n$-tuple $(\alpha_1, \cdots, \alpha_{n-s}, \underbrace{\pi, \cdots, \pi}_{s})$, with $\alpha_{n-s} < \pi, \ s= 0, 1, \cdots , n$ and $\zeta \in \big[\!-\!\lfloor \dfrac{n}{2} \rfloor , \lfloor \dfrac{n}{2} \rfloor\,\big] \cap \,\mathbb{Z}\,.$
Also note that, if $\Psi$ is the automorphism of the vector space $\mathfrak{su}_n$ \,defined by $\Psi(X):=-X$ (for any $X \in \mathfrak{su}_n$), we have $\Psi\big(\Theta(Q^*)\big)=\Theta(Q)\,.$

\smallskip

d) The map \ $Q \mapsto m(Q)=d(Q, I_{_n})^2$ \ is continuous on the compact Lie group $SU_n$, so it has a maximum \ $\delta_n$ \ on $SU_n$. Since $(SU_n, \phi)$ is a homogeneous Riemannian manifold, it is clear that \ $\delta_n=\max \{m(Q): Q \in SU_n\}$ agrees with the square of the diameter of $SU_n$ with respect to $d$, \ i.e. \ $\delta(SU_n, \phi)=\sqrt{\delta_n}\ .$ \ Keeping in mind (b), (c) above, Remarks \ref{prelim-Q-q} (c) and Remark-Definition \ref{admissibility}, we deduce that
\ $\delta_n$ \ is the maximum of the set \ $\big\{m(\alpha_1, \cdots, \alpha_n): (\alpha_1, \cdots, \alpha_n) \in \mathbb{R}^n \mbox{\ is any $\zeta$-admissible $n$-tuple, with\ \,} \zeta \in \big[0 , \lfloor \dfrac{n}{2} \rfloor\,\big]\cap \mathbb{Z} \big\}\,.$
\end{remsdefs}

\begin{rem}\label{Adjoint-map-equalities}
For any $Q\in SU_n$ and $U\in U_n\,,$ we have $m\big(Ad_{_U}(Q)\big)=m(Q)$ and $\Theta\big(Ad_{_U}(Q)\big)=Ad_{_U}\big(\Theta(Q)\big)\,.$

The equality $m\big(Ad_{_U}(Q)\big)=m(Q)$ holds because the function $m(Q)$ depends only on the eigenvalues of $Q\,,$ while the second equality $\Theta\big(Ad_{_U}(Q)\big)=Ad_{_U}\big(\Theta(Q)\big)\,$ can be easily deduced from the first one,  remembering that the map $Ad_{_U}: (SU_n,\phi) \to (SU_n,\phi)$ is an isometry that commutes with the exponential map.

\end{rem}

\begin{lemma}\label{casi-zeta-intermedi}
Let $Q$ be a matrix of $SU_n$ $(n\ge 2)$, whose $n$ eigenvalues are $\mu_1, \cdots , \mu_n\,.$ Denote by \,$s(Q)$  the multiplicity of \ $-1$  as an eigenvalue of \,$Q$ \ and let \\ 
$\zeta(Q)=\dfrac{1}{2 \pi} \sum\limits_{j=1}^n \arg(\mu_j)\,,  
\ \ \  m(Q)=\min\{ \Vert X \Vert_{_\phi}^2 \ : \ X \in \mathfrak{su}_n \mbox{ and } \exp(X)=Q \}\,,$

$\Theta(Q)= \{X\in \mathfrak{su}_n : \exp(X)=Q\,, \  \Vert X \Vert_{_\phi}^2= m(Q)\}$. 

Then the following facts are equivalent:

i) $\mathfrak{su}_n$\!--\,$plog (Q) \ne \emptyset$;

ii) $\zeta(Q) \in \{0, 1, \cdots , s(Q)\}$;

iii) $m(Q) = \sum\limits_{j=1}^n \big(\!\arg(\mu_j)\big)^2$.

Moreover, if any of the above conditions holds,  then we have \ $\Theta(Q)=\mathfrak{su}_n$\!--\,$plog (Q)$.
\end{lemma}

\begin{proof}
The equivalence between (i) and (iii) and the final assertion follow directly from the proof of \cite[Prop.\,5.5\,(b)]{PD2022}. 

Set $s:= s(Q)$ and $\zeta:=\zeta(Q)$. We can assume $\ \arg(\mu_1) \le \arg(\mu_2) \le \cdots \le \arg(\mu_{n})\,;$  so we have \,$\mu_{n-s+1}= \mu_{n-s+2}= \cdots =\mu_n =-1$, while \ $\mu_1, \cdots , \mu_{n-s}$ \ are all different from $-1$.

Assume now (i) and fix a matrix $X \in \mathfrak{su}_n$\!--\,$plog (Q)$ so that  the $n$ eigenvalues of $X$ are: $\arg(\mu_1){\bf i}, \cdots , \arg(\mu_{n-s}){\bf i}$, \ \  $-\pi {\bf i} $ with multiplicity $k \ge 0$ and $\pi {\bf i}$ with multiplicity $s-k \ge 0$, for some $k \in \{0, \cdots , s\}$. The condition $\mbox{tr}(X)=0$ implies $\sum\limits_{j=1}^{n-s} \arg(\mu_j) + (s -2k) \pi =0 $. Since $\zeta= \dfrac{1}{2\pi}\sum\limits_{j=1}^n \arg(\mu_j) = \dfrac{1}{2\pi} \big( \sum\limits_{j=1}^{n-s} \arg(\mu_j) + s \pi\big)$, we obtain $\zeta=k \in \{0,  1 , \cdots, s\}$ and so (i) implies (ii).

Assume now (ii) (i.e. $\zeta \in \{0, 1, \cdots , s\}$). By Remarks-Definition \ref{def-m(Q} (a), we have

$m(Q) = \min \{  \sum\limits_{j=1}^n \big(\!\arg (\mu_j) + 2 k_j \pi\big)^2 \ : \ k_1, \cdots k_n \in \mathbb{Z} \ \mbox{ and } \  \sum\limits_{j=1}^n  k_j  =-\zeta\}$. \\Choose $h_1 = \cdots h_{n-\zeta} =0$ and $h_{n-\zeta+1} = h_{n-\zeta+2} = \cdots = h_n = -1$, so that $\sum\limits_{j=1}^n  h_j  =-\zeta$. Since $\zeta \le s$, we get $\sum\limits_{j=1}^n \big(\!\arg (\mu_j) + 2 h_j \pi\big)^2 = \sum\limits_{j=1}^{n-\zeta} \big(\!\arg(\mu_j)\big)^2 + \zeta(-\pi)^2 = \sum\limits_{j=1}^n \big(\!\arg(\mu_j)\big)^2$. This gives $m(Q) \le \sum\limits_{j=1}^n \big(\!\arg(\mu_j)\big)^2$. The equality follows from the inequality (*) of Remarks-Definitions \ref{def-m(Q}\,(a).  Thus (ii) implies (iii) and the proof  is complete.
\end{proof}

\begin{lemma}\label{min-su-W0}
With the same notations as in Remarks-Definitions  \ref{def-m(Q} (a), we have

$\psi_{_Q}(\underline{k}) > m(Q)$  \ for every $\underline{k} \in W_Q \setminus Z_Q$.\ \ Therefore \ 
$m(Q) = \min \{ \psi_{_Q}(\underline{k}) \ : \  \underline{k} \in Z_Q\}$.
\end{lemma}

\begin{proof} Remembering that $m(Q) = \min \{ \psi_{_Q}(\underline{k}) \ : \  \underline{k} \in W_Q\}$, it suffices to prove that, if $\underline{k} \in W_Q \setminus Z_Q$, then there exists $\underline{l} \in W_Q$ such that $\psi_{_Q}(\underline{k}) > \psi_{_Q}(\underline{l})$.

Let $\underline{k} = (k_1, \cdots , k_n)  \in W_Q \setminus Z_Q$, \ \ i.e. $k_1, \cdots , k_n \in \mathbb{Z}$, \  $ \sum\limits_{j=1}^n k_j = -\zeta(Q)$ \ and \ $\Delta_{_Q}(\underline{k}) \ge 2\,.$

Fix two indices $r, t \in \{1, \cdots, n\}$ such that $k_t = \max\{k_1, \cdots , k_n\}$ and $k_r = \min\{k_1, \cdots , k_n\}$. From the definition we have $k_t-k_r = \Delta_{_Q}(\underline{k}) \ge 2$. 

We define $l_j:=k_j$ for every $j \in \{1, \cdots, n\}\setminus\{r, t\}$, \ \,$l_r:= k_r+1$ ,\  \ $l_t:= k_t-1$ and 
$\underline{l} :=(l_1, \cdots , l_n)\in \mathbb{Z}^n$. 
\ Since
\ $\sum\limits_{j=1}^n l_j=\sum\limits_{j=1}^n k_j =- \zeta (Q)\,,$\ we have \ $\underline{l} \in W_Q\,.$

The inequality $\psi_{_Q}(\underline{k}) >  \psi_{_Q}(\underline{l})$ is equivalent to 

$(\arg(\mu_r) +2k_r \pi)^2 + (\arg(\mu_t) +2k_t \pi)^2 > (\arg(\mu_r) +2l_r \pi)^2 + (\arg(\mu_t) +2l_t \pi)^2$ and the latter is equivalent to the inequality 
$2(k_t-k_r)\pi > 2 \pi + \arg(\mu_r) - \arg(\mu_t)\,$, which is satisfied since $2 (k_t-k_r)\pi \ge 4 \pi$ \ and \ $\arg(\mu_r) -\arg(\mu_t) < 2 \pi\,.$ 
 \ Hence we have \ $\psi_{_Q}(\underline{k}) >  \psi_{_Q}(\underline{l})$ \ and this concludes the proof.
\end{proof}

\begin{lemma}\label{min-0-1}
With the same notations as in Remarks-Definitions  \ref{def-m(Q} (a), let 

$\underline{k}=(k_1, \cdots , k_n) \in Z_Q$.

If $\zeta(Q) \ge 0$, then $\underline{k}$ has  $\zeta(Q)$ entries equal to $-1$, while the remaining ones are equal to $0$. In particular, if $\zeta(Q)=0$, then $Z_Q$ consists of the unique element $(0, \cdots , 0)$.
\end{lemma}

\begin{proof}
Set $\zeta : = \zeta(Q)$. Since $\Delta_{_Q}(\underline{k}) \le 1$, there exists an integer $H$ and a non-empty set $J \subseteq \{1, \cdots , n\}$ such that $k_j = H$ for every $j \in J$ and $k_i = H-1$ for every $i \in \{1, \cdots , n\} \setminus J$. Let $\chi \ge 1$ be the cardinality of $J$. Then $-\zeta= \sum\limits_{j=1}^n k_j = \chi H + (n-\chi)(H-1) = n(H-1)+\chi\,.$ Remembering Remarks \ref{prelim-Q-q} (c), we get $n(H-1)= -\zeta -\chi \ge - \lfloor \dfrac{n}{2} \rfloor - n > -2n\,:$ so $H \ge 0$.
Since $\zeta \ge 0$, we have $n(H-1) = -\zeta - \chi \le -1$, hence $H < 1\,$ and so $H=0\,.$ Therefore $\underline{k}$ necessarily has  $\zeta$ entries equal to $-1$, while the remaining ones are equal to $0$.
\end{proof}

\begin{prop}\label{m-valutation}
Let $Q$ be a matrix of $SU_n$ ($n\ge 2$) whose $n$ eigenvalues are $\mu_1, \cdots , \mu_n$ ordered so that we have $\ \arg(\mu_1) \le \arg(\mu_2) \le \cdots \le \arg(\mu_{n})\,$ and assume that\\ $\zeta(Q)=\dfrac{1}{2 \pi} \sum\limits_{j=1}^n \arg(\mu_j)\ge 0\,.$ Let \, 
$m(Q)=\min\{ \Vert X \Vert_{_\phi}^2 \ : \ X \in \mathfrak{su}_n \mbox{ and } \exp(X)=Q \}$ 
and $\Theta(Q)= \{X\in \mathfrak{su}_n : \exp(X)=Q\,, \  \Vert X \Vert_{_\phi}^2= m(Q)\}$.

Finally, let $X_0$ be any $\mathfrak{su}_n$--logarithm of $Q$ \ whose $n$ eigenvalues are : \\ $(\arg(\mu_1) +2 h_1 \pi)$ $\!\!\bf{i}$  $ , \cdots , (\arg(\mu_n)+ 2 h_n\pi) \bf{i}$ \ $(\!\!\!$ with\ \,$h_1, \cdots, h_n \in \mathbb{Z}$\ and \ $\sum\limits_{j=1}^n h_j = -\zeta(Q))\,. $

a) If \ $\zeta(Q)=0 $, \ then \ $m(Q) = \sum\limits_{j=1}^{n} \big(\!\arg(\mu_j)\big)^2\,;$ 

furthermore, \ $X_0 \in \Theta(Q)$ \ if and only if 
 \ $ h_j=0\,$\ for every \ $j\in \{1,\cdots, n\}\,.$

\smallskip

b) If \ $\zeta(Q) \ge 1\,,$  
\ then \ $m(Q) = \sum\limits_{j=1}^{n-\zeta(Q)} \big(\!\arg(\mu_j)\big)^2 + \sum\limits_{j=n-\zeta(Q)+1}^{n} \big(2\pi-\arg(\mu_j)\big)^2$;
\smallskip

furthermore, if \,$\mu_{n-\zeta(Q)} \ne \mu_{n-\zeta(Q) +1}\,,$ we have \ $X_0 \in \Theta(Q)$ \ if and only if 

$ h_j=0\,$\ for every \ $j\in \{1,\cdots, n-\zeta(Q)\}\,$\ and \ $h_l=-1\,$\ for every \ $l\in \{n-\zeta(Q)+1,\cdots, n\}\,;$ 

while, if \ $\mu_{n-\zeta(Q)}=\mu_{n-\zeta(Q) +1}$, \ we have \ $X_0 \in \Theta(Q)$ \ if and only if 

$h_r=0$\ for every $r\in \{1,\cdots, n-\zeta(Q)-1\}$ such that $\mu_r\ne \mu_{n-\zeta(Q)}$, 

$h_t=-1$\ for every  
$t\in \{n-\zeta(Q)+2,\cdots, n\}$ such that $\mu_t\ne \mu_{n-\zeta(Q)}$,

and 
\ $h_m=0$ or $h_m=-1\,$ for every index $m$ such that $\mu_m= \mu_{n-\zeta(Q)}\,$ 

$($with the constraint that the equality $\sum\limits_{j=1}^n h_j = - \zeta(Q)\,$ is satisfied $).$

\end{prop}

\begin{proof}
Set $\zeta:=\zeta(Q)\,.$

If $\zeta = 0\,,$ the formula for $m(Q)$ follows directly from Lemma \ref{casi-zeta-intermedi}\,. Clearly, if \\ $h_1=\cdots=h_n=0\,$ we have $\Vert X_0 \Vert_{_\phi}^2= m(Q)\,,$ and so \ $X_0 \in \Theta(Q)\,.$ Conversely, if \ $X_0 \in \Theta(Q)$ then, by Lemma \ref{min-su-W0}\,, \,$(h_1, \cdots, h_n) \in Z_Q\,;$
hence, by Lemma \ref{min-0-1}\,, we obtain $h_j=0\,$ for every \ $j\in \{1,\cdots, n\}\,$ and the proof of (a) is complete.
\smallskip

Now let $\zeta \ge 1\,,$ and denote $\underline{k}_0 :=(\, \underbrace{0, \cdots , 0}_{n-\zeta}, \, \underbrace{-1, \cdots, -1}_{\zeta} \, )\in Z_Q \subset W_Q\,.$ Remembering Remarks-Definitions \ref{def-m(Q} (a), we want to prove that the minimum of the map $\psi_{_Q}$ on $W_Q$ is reached at $\underline{k}_0\,.$ By Lemma \ref{min-su-W0}, this minimum is reached only on $Z_Q$, while, by Lemma \ref{min-0-1}, every   $\underline{k} \in Z_Q$ has  $\zeta$ entries equal to $-1$ and the remaining ones equal to $0$. Let $\underline{k}=(k_1, \cdots, k_n) \in Z_Q$ and assume that there are two indices $1\le i < r \le n$ such that $k_i=-1$ and $k_{r}=0$; then 
we define a new element $\underline{l} = (l_1, \cdots , l_n) \in Z_Q$ \,with\, $l_j= k_j$ for any $j\ne i, r,\ \ l_i = 0,\ \ l_{r} = - 1$, and we obtain $\psi_{_Q} (\underline{k}) \ge \psi_{_Q}(\underline{l})\,.$ Indeed this inequality is equivalent to
$\big(\!\arg (\mu_i) -2\pi\big)^2 + \big(\!\arg(\mu_r)\big)^2 \ge \big(\!\arg(\mu_i)\big)^2 + \big(\!\arg (\mu_{r}) -2\pi\big)^2$, and the latter inequality is equivalent to 
\ $\arg (\mu_i) \le \arg (\mu_{r})\,,$ which is satisfied since $i<r\,.$ Also note that we have $\psi_{_Q} (\underline{k}) = \psi_{_Q}(\underline{l})\,$ if and only if $\arg (\mu_i) =\arg (\mu_{r})\,,$ i.e. if and only if $\mu_i=\mu_r\,.$
Repeating, if necessary, on $\underline{l}$ the same operation performed on $\underline{k}$, after a finite number of steps we obtain $\psi_{_Q}(\underline{k}) \ge \psi_{_Q}(\underline{k}_0)$, for any $\underline{k} \in Z_Q\,.$ \ It follows that we have $m(Q)=\psi_{_Q}(\underline{k}_0)= \sum\limits_{j=1}^{n-\zeta} \big(\!\arg(\mu_j)\big)^2 + \sum\limits_{j=n-\zeta+1}^{n} \big(2\pi-\arg(\mu_j)\big)^2\,.$
We also obtain that, 
if 
$\mu_{n-\zeta}\ne \mu_{n-\zeta+1}\,,$ the equality $\psi_{_Q}(\underline{k}) = \psi_{_Q}(\underline{k}_0)$ holds if and only if $\underline{k}=\underline{k}_0\,;$ while, if $\mu_{n-\zeta}=\mu_{n-\zeta+1}\,,$ we have $\psi_{_Q}(\underline{k}) = \psi_{_Q}(\underline{k}_0)$ if and only if $k_r=0$\ for every $r\in \{1,\cdots, n-\zeta-1\}$ such that $\mu_r\ne \mu_{n-\zeta}\,$, 
\ $k_t=-1$\ for every  
$t\in \{n-\zeta+2,\cdots, n\}$ such that $\mu_t\ne \mu_{n-\zeta}\,$,
and 
\ $k_m=0$ or $k_m=-1\,$ for every $m$ such that $\mu_m= \mu_{n-\zeta}\,.$ 
Thus the second part of statement (b) has also been proved.
\end{proof}

\begin{rem}
Note that, if $0 \le \zeta(Q) \le s(Q)$, the formulas in part (a) and in part (b) of Proposition \ref{m-valutation}, reduce to formula \ (iii) of Lemma  
 \ref{casi-zeta-intermedi}.
\end{rem}

\section{The set $\Theta(Q)$ of minimizing logarithms of any special unitary matrix $Q$}

\begin{rem}\label{conseguenze}
From Proposition \ref{m-valutation} and with the same notations and hypotheses, we get: 

\ \ i) if $\zeta(Q)=0\,,$ \ then 
$\Theta(Q)$ agrees with the set of all $\mathfrak{su}_n\mbox{--logarithms of } Q$ whose $n$ eigenvalues are : $\arg(\mu_1)$ $\!\!\bf {i}\,,$ $\arg(\mu_2)$ $\!\!\bf {i}\,,$ $\cdots, \arg(\mu_{n-1})$ $\!\!\bf {i}\,,$ $ \arg(\mu_n)$ $\!\!\bf {i}\,;$

\ \ ii) if $\zeta(Q) \ge 1\,,$ \  then $\Theta(Q)$ agrees with the set of all $\mathfrak{su}_n\mbox{--logarithms of } Q$ whose $n$ eigenvalues are the following: 

$\arg(\mu_1)$ $\!\!\bf {i}\,,$ $\arg(\mu_2)$ $\!\!\bf {i}\,,$ $\cdots, \arg(\mu_{n-\zeta(Q)-1})$ $\!\!\bf {i}\,,$ $ \arg(\mu_{n-\zeta(Q)})$ $\!\!\bf {i}\,,$ $ \big(\!\arg(\mu_{n-\zeta(Q)+1})-2\pi\big)$ $\!\!\bf {i}\,,$ \\ $\big(\!\arg(\mu_{n-\zeta(Q)+2})-2\pi\big)$ $\!\!\bf {i}\,,$ $\cdots, \big(\!\arg(\mu_{n-1})-2\pi\big)$ $\!\!\bf {i}\,,$ $ \big(\!\arg(\mu_{n})-2\pi\big)$ $\!\!\bf {i}\,$. 
\end{rem}

\begin{prop}\label{Theta(Q)}
Let $Q$ be a matrix of $SU_n$ ($n\ge 2$) whose $n$ eigenvalues are $\mu_1, \cdots , \mu_n$ ordered so that we have $\ \arg(\mu_1) \le \arg(\mu_2) \le \cdots \le \arg(\mu_{n})\,$ and assume that\\ $\zeta(Q)=\dfrac{1}{2 \pi} \sum\limits_{j=1}^n \arg(\mu_j)\ge 0\,.$ \,Let \, 
$m(Q)=\min\{ \Vert X \Vert_{_\phi}^2 \ : \ X \in \mathfrak{su}_n \mbox{ and } \exp(X)=Q \}$ 
and  \\ $\Theta(Q)= \big\{X\in \mathfrak{su}_n : \exp(X)=Q\,, \  \Vert X \Vert_{_\phi}^2= m(Q)\big\}\,.$ \\ If $\zeta(Q) \ge 1$, we also set \ $\nu_{_1}:=max\big\{j\in \{1,\cdots, n-\zeta(Q)\}: \mu_{n-\zeta(Q)+1-j}=\mu_{n-\zeta(Q)}\big\}$ \ and
\ $\nu_{_2}:=max\big\{j\in \{1,\cdots, \zeta(Q)\}: \mu_{n-\zeta(Q)+j}=\mu_{n-\zeta(Q)+1}\big\}\,.$ 

a) If $\zeta(Q)=0$ \,, then $\Theta(Q)$ is a single point.

b) If $\zeta(Q)\ge 1$ and $\mu_{n-\zeta(Q)} \ne \mu_{n-\zeta(Q)+1}$\,, then $\Theta(Q)$ is a single point.

c) If $\zeta(Q)\ge 1$ and $\mu_{n-\zeta(Q)}=\mu_{n-\zeta(Q)+1}$\,, then $\Theta(Q)$ is a compact submanifold of $\mathfrak{su}_n$ diffeomorphic to the symmetric space $\dfrac{U_{(\nu_{_1}+\nu_{_2})}}{U_{\nu_{_1}}\oplus U_{\nu_{_2}}}$ and therefore also to the complex Grassmannian ${\bf Gr}(\nu_{_2}; \mathbb{C}^{(\nu_{_1}+\nu_{_2})})\,.$
\end{prop}

\begin{proof}
Set $\beta:=\mu_{n-\zeta(Q)}\,.$  We denote by $S_1$ the set consisting of all distinct eigenvalues $\lambda$ of $Q$ such that $\arg(\lambda)<\arg(\beta)$ and by $S_2$ the set consisting of all distinct eigenvalues $\varphi$ of $Q$ such that $\arg(\varphi)>\arg(\beta)\,$ ($S_1$ and $S_2$ can also be empty); for every eigenvalue $\epsilon$ of $Q\,,$ we also denote by $m(\epsilon)\ge1$ the multiplicity of $\epsilon$ as an eigenvalue of $Q\,.$ If we set $J:=\big(\bigoplus\limits_{\lambda \in S_1} \lambda I_{_{m(\lambda)}}\big)\oplus \big(\beta I_{_{m(\beta)}}\big)\oplus \big(\bigoplus\limits_{\varphi \in S_2} \varphi I_{_{m(\varphi)}}\big)\,,$ it is a well-known fact that there exists $U \in U_n$ such that $Q=Ad_{_U}(J)\,;$ so, by Remark \ref{Adjoint-map-equalities} , we can assume that \\ $Q=J=\big(\bigoplus\limits_{\lambda \in S_1} \lambda I_{_{m(\lambda)}}\big)\oplus \big(\beta I_{_{m(\beta)}}\big)\oplus \big(\bigoplus\limits_{\varphi \in S_2} \varphi I_{_{m(\varphi)}}\big)\,.$

Now, if $\zeta(Q)=0$ as in (a) or if $\zeta(Q)\ge 1$ and $\beta=\mu_{n-\zeta(Q)} \ne \mu_{n-\zeta(Q)+1}\,$ as in (b),  we set

$\widehat{J}:=\big(\bigoplus\limits_{\lambda \in S_1} \arg(\lambda)$ $\!\!\bf {i}$ $I_{_{m(\lambda)}}\big)\oplus \big(\arg(\beta)$ $\!\!\bf {i}$ $I_{_{m(\beta)}}\big)\oplus \big(\bigoplus\limits_{\varphi \in S_2} (\arg(\varphi)-2\pi)$ $\!\!\bf {i}$ $I_{_{m(\varphi)}}\big)\,.$ 

(Here we agree that, if $S_1$ or $S_2$ are empty, the terms corresponding to them do not appear in all previous direct sums.)

Note that $\exp(\widehat{J})=J=Q\,.$
By Remark \ref{conseguenze}, a matrix $X\in \mathfrak{su}_n$ belongs to $\Theta(Q)$ if and only if there is $R\in U_n$ such that $X=Ad_{_R}(\widehat{J})$ \ and \ $Q=\exp(X)=Ad_{_R}(\exp(\widehat{J}))=Ad_{_R}(Q)$, i. e. if and only if $X=Ad_{_R}(\widehat{J})$ \ for some $R\in U_n$ such that $RQ=QR\,.$ By Lemma \ref{commutazione}, a matrix $R\in U_n$ commutes with $Q$ if and only if $R\in \big(\bigoplus\limits_{\lambda \in S_1} U_{m(\lambda)}\big)\oplus U_{m(\beta)}\oplus \big(\bigoplus\limits_{\varphi \in S_2} U_{m(\varphi)}\big)\,;$ this implies that $R$ also commutes with $\widehat{J}\,,$ and so $X=Ad_{_R}(\widehat{J})=\widehat{J}\,.$ \ Hence  $\widehat{J}$ is the unique element of $\Theta(Q)\,.$ This proves (a) and (b).

Now, assume as in (c): \,$\zeta(Q)\ge 1$ \,and \,$\beta=\mu_{n-\zeta(Q)}= \mu_{n-\zeta(Q)+1}\,,$ so that $m(\beta)=\nu_{_1}+\nu_{_2}\,.$ 

If we set
$\widetilde{J}:=$ 

$\big(\bigoplus\limits_{\lambda \in S_1} \arg(\lambda)$ $\!\bf {i}$ $I_{_{m(\lambda)}}\big)\oplus \big(\arg(\beta)$ $\!\bf {i}$ $I_{_{\nu_{_1}}}\big)\oplus \big((\arg(\beta)-2\pi)$ $\!\bf {i}$ $I_{_{\nu_{_2}}}\big)\oplus \big(\bigoplus\limits_{\varphi \in S_2} (\arg(\varphi)-2\pi)$ $\!\bf {i}$ $I_{_{m(\varphi)}}\big)\,,$ 

we have $\exp(\widetilde{J})=J=Q\,;$  as in the proof of parts (a) and (b) and using Lemma \ref{commutazione} and Remark \ref{conseguenze} again, we can prove  that a matrix $X\in \mathfrak{su}_n$ belongs to $\Theta(Q)$ if and only if  $X=Ad_{_R}(\widetilde{J})$ for some $R\in U_n$ commuting with $Q\,,$ i.e. if and only if $X=Ad_{_R}(\widetilde{J})$ for some 
$R\in G:=\big(\bigoplus\limits_{\lambda \in S_1} U_{m(\lambda)}\big)\oplus U_{(\nu_{_1}+\nu_{_2})}\oplus \big(\bigoplus\limits_{\varphi \in S_2} U_{m(\varphi)}\big)\,.$
Since the map $(R,X) \mapsto Ad_{_R}(X)$ defines a left action of the compact Lie group $G$ on $\mathfrak{su}_n\,,$  we conclude that  $\Theta(Q)$ agrees with the orbit of $\widetilde{J}$ with respect to this action, while the corresponding isotropy subgroup at $\widetilde{J}$ \,is
\,$\widehat{G}:=\big(\bigoplus\limits_{\lambda \in S_1} U_{m(\lambda)}\big)\oplus U_{\nu_{_1}}\oplus U_{\nu_{_2}}\oplus \big(\bigoplus\limits_{\varphi \in S_2} U_{m(\varphi)}\big)\,.$ Hence we conclude that $\Theta(Q)$ is a compact submanifold of $\mathfrak{su}_n$ diffeomorphic to $\dfrac{G}{\widehat{G}}\cong\dfrac{U_{(\nu_{_1}+\nu_{_2})}}{U_{\nu_{_1}}\oplus U_{\nu_{_2}}}\,$
(see, for instance, \cite{EoM-Orbit}). Since it is known that $\dfrac{U_{(\nu_{_1}+\nu_{_2})}}{U_{\nu_{_1}}\oplus U_{\nu_{_2}}}\,$ is a symmetric space diffeomorphic to the complex Grassmannian ${\bf Gr}(\nu_{_2}; \mathbb{C}^{(\nu_{_1}+\nu_{_2})})\,,$ the proof of the Proposition is complete.
\end{proof}

\begin{prop}\label{sunlog}
Let $Q$ be a matrix of $SU_n$ ($n\ge 2$) whose $n$ eigenvalues are $\mu_1, \cdots , \mu_n\,.$ Denote by \,$s(Q)$  the multiplicity of \ $-1$  as an eigenvalue of \,$Q$ \ and let \\ 
$\zeta(Q)=\dfrac{1}{2 \pi} \sum\limits_{j=1}^n \arg(\mu_j)\,.$\ \ Then

a) if $\zeta(Q)>s(Q)$ or $\zeta(Q)<0\,,$ \ the set $\mathfrak{su}_n$\!--\,$plog (Q)$ is empty;

b) if $0\le \zeta(Q) \le s(Q)\,,$ \ the set \ $\mathfrak{su}_n$\!--\,$plog (Q)$ is a compact submanifold of $\mathfrak{su}_n$ diffeomorphic to the complex Grassmannian ${\bf Gr}(\zeta(Q); \mathbb{C}^{s(Q)})\,;$ \ in particular, $\mathfrak{su}_n$\!--\,$plog (Q)$ is a single point if and only if \, $\zeta(Q)=0$ \ or \ $\zeta(Q)=s(Q)\,.$
\end{prop}

\begin{proof}
The proof follows directly from Lemma \ref{casi-zeta-intermedi}\, and Proposition \ref{Theta(Q)}\,.
\end{proof}

\section{About diameter and diametral pairs of $(SU_n, \phi)$}

\begin{prop}\label{diameter}
Let $m: SU_n \to \mathbb{R}$  be the map defined in Remarks-Definitions \ref{def-m(Q} (a) and let $\delta_n=\max \{m(Q): Q \in SU_n\}$ \ $(n \ge 2)\,.$ \ Then we have\\
$\delta_n=$
$\left\lbrace
\begin{array}{cc}
\ \ \ \ n\pi^2 \ \ \ \ \ \ \ \ \ \ \ \ \ \ \ \ \mbox{if\ $n$\ is\ even\,}\\
\!\!(n-\dfrac{1}{n}) \pi^2 \ \ \ \ \ \ \ \ \ \ \ \mbox{if\ $n$\ is\ odd}
\end{array}\ \ ;
\right. $  
\ \ \ 

furthermore a matrix $Q \in SU_n$ satisfies
$
m(Q)=\delta_n$ \ if and only if \ \ 

$\left\lbrace
\begin{array}{cc}
Q=-I_{_n} \ \ \ \ \ \ \ \ \ \ \ \ \ \ \ \ \ \ \ \ \ \ \ \ \ \ \ \ \ \ \ \ \ \ \ \ \ \ \ \ \ \ \ \mbox{if\ $n$\ is\ even\,}\\
\!\!\!Q=e^{\frac{(n-1)\pi\bf {i}}{n}} I_{_n} \mbox{\ \ or\ \ } Q=e^{\frac{-(n-1)\pi\bf {i}}{n}} I_{_n} \ \ \ \ \ \ \ \ \ \mbox{if\ $n$\ is\ odd}
\end{array}.
\right. $  
\end{prop}

\begin{proof}
As noted in Remarks-Definitions \ref{def-m(Q} (d), \ $\delta_n$  agrees with the absolute maximum of the map \ $(\alpha_1, \cdots, \alpha_n) \mapsto m(\alpha_1, \cdots, \alpha_n)$, where $(\alpha_1, \cdots, \alpha_n)$ varies over the set of all $\zeta\mbox{-admissible\ } n$-tuples of real numbers, with $\zeta \in \big[0 , \lfloor \dfrac{n}{2} \rfloor\,\big]\cap \mathbb{Z}\ .$ 
Now, fix any $\zeta\mbox{-admissible}$ $n$-tuple  $(\beta_1, \cdots, \beta_n)$ such that $m(\beta_1, \cdots, \beta_n)=\delta_n\,$ (where $\dfrac{1}{2 \pi} \sum\limits_{j=1}^n \beta_j =\zeta \in \big[0 , \lfloor \dfrac{n}{2} \rfloor\,\big]\cap\, \mathbb{Z})\,.$ 

\smallskip

First of all, we will prove that we necessarily have $(\beta_1, \cdots, \beta_n)=(\,\underbrace{\pi, \cdots, \pi}_{n}\,)$  \,when $n$ is even\ \,and \ $(\beta_1, \cdots, \beta_n)=(\,\underbrace{\dfrac{(n-1)\pi}{n}, \cdots, \dfrac{(n-1)\pi}{n}}_{n}\,)$  \,when $n$ is odd, and then we will prove all the statements of the Proposition. 
\ The proof will be done in several steps.

i) \ \ If $\zeta \ge 1\,,$ at least one of the following two conditions is necessarily true:

(ia) \ \ $\beta_1=\beta_2=\cdots=\beta_{n-\zeta}=\beta_{n-\zeta+1}\ ;$

(ib) \ \ $\beta_{n-\zeta}=\beta_{n-\zeta+1}=\cdots=\beta_{n-1}=\beta_n\ ;$ 

while, if $\zeta=0\,,$ condition (ib) is obviously always satisfied.

In the case $\zeta \ge 1\,,$ if both conditions (ia) and (ib) are false, then the sets 

$J_1:= \big\{i\in \{1, \cdots,  n-\zeta \}: \beta_i < \beta_{n-\zeta+1}  \big\}\ , \  
J_2:= \big\{j\in \{n-\zeta +1, \cdots, n \}: \beta_j > \beta_{n-\zeta} \big\}$ are both non-empty, and therefore we can define \ $h:=\max J_1 \ , \ \ k:=\min J_2\,.$  \,Note that, from definitions of $h$ and $k$, we have $\beta_h < \beta_{h+1}$ and \,$\beta_{k-1} < \beta_k$\,; so we can choose $\epsilon >0$ such that $\beta_h +\epsilon < \beta_{h+1}$ and $\beta_{k-1} < \beta_k - \epsilon$\,. Now we consider the new $n$-tuple $(\gamma_1, \cdots, \gamma_n)$, where $\gamma_r := \beta_r$ when $r\ne h, k\,,\ \gamma_h:=\beta_h + \epsilon , \  \gamma_k:= \beta_k - \epsilon$\,; we have $\sum\limits_{j=1}^n \gamma_j=\sum\limits_{j=1}^n \beta_j$\,, so the $n$-tuple $(\gamma_1, \cdots, \gamma_n)$ is also $\zeta-\mbox{admissible}$. Since $\beta_k-\beta_h<2\pi$, by Proposition \ref {m-valutation} (b), we get \,$m(\gamma_1, \cdots, \gamma_n)-\delta_n=\gamma_h^{\,2} + (2\pi - \gamma_k)^2-\big[\beta_h^{\,2} + (2\pi- \beta_k)^2\big]=2\epsilon\big(\epsilon+2\pi-(\beta_k-\beta_h)\big)>0\,,$ and this is contrary to the definition of $\delta_n$\,. So we conclude that (ia) or (ib) is true.

ii) \ \ If condition (ia) holds (with $\zeta \ge 1$), \ then \ $\beta_1=\beta_2=\cdots=\beta_{n-1}=\beta_n\,.$

In fact, otherwise, there is an index $r \in \big\{n-\zeta+1,\, n-\zeta+2, \cdots, n-1, n\big\}$\,, such that \\$\beta_{r+1} > \beta_{r}=\beta_{r-1}=\cdots=\beta_{n-\zeta+1}=\beta_{n-\zeta}=\cdots=\beta_1\,.$ Then, after setting $w:=\beta_{r+1}\,\mbox{\ and\ }\\\,z:=\beta_r=\cdots=\beta_1$ (so that we have \ $r\,z= 2\pi\zeta -w -\sum\limits_{j=r+2}^n \beta_j$), we choose $\epsilon >0$ such that $w -r\cdot\epsilon > z +\,\epsilon\,,$ and we consider the $\zeta\mbox{-admissible} $ $n$-tuple $(\varphi_1, \cdots, \varphi_n)$ defined by \\ $\varphi_j:=z +\,\epsilon$ \ for  $j=1, \cdots, r\,,\ \varphi_{r+1}:= w - r\cdot\epsilon\,$ and $\varphi_h:=\beta_h$\, for \,$h=r+2, \cdots, n\,.$ \\ Taking into account Proposition \ref{m-valutation} (b), \ we obtain \ $m(\varphi_1, \cdots, \varphi_n) -\delta_n=\\ (n-\zeta) \big[(z+ \epsilon)^2 -z^2\big] +(r-n+\zeta)\big[(2\pi-z -\epsilon)^2-(2\pi-z)^2 \big]+ (2\pi-w +r\,\epsilon)^2- (2\pi-w)^2=\,\epsilon^2(r+ r^2 )+2\epsilon\big[(n-\zeta)z-(r-n+\zeta)(2\pi-z)+r(2\pi-w)\big]=\,\epsilon^2(r+ r^2 )+ 2\epsilon\big[r\,z +2\pi (n-\zeta)- r\,w\big] = \epsilon^2(r +r^2)+ 2\epsilon\big[2\pi n- (r+1)w -\sum\limits_{j=r+2}^n \beta_j\big]\,.$ Since \ $(r+1)w +\sum\limits_{j=r+2}^n \beta_j\le n\pi\,,$ then we get \,$m(\varphi_1, \cdots, \varphi_n) -\delta_n\ge \epsilon^2(r +r^2) + 2\epsilon n \pi\,>0\,,$ and this is impossible, bearing in mind the definition of $\delta_n$\,. So we conclude that (ii) holds.

iii) \ \ If condition (ib) holds (now also with $\zeta \ge 0$), \ then \ $\beta_2=\beta_3=\cdots=\beta_{n-1}=\beta_n\,.$

Otherwise, we have $K:=\big\{j\in \{2, 3, \cdots, n-\zeta-1\}: \beta_j<\beta_{n-\zeta}  \big\} \ne \emptyset\,,$ and so, called $t:=\max K\,,$ we fix a positive real number $\epsilon$ such that \,$-\pi< \beta_1 -\epsilon\,, \ \,\beta _t +\epsilon < \beta_{t+1}=\beta_{n-\zeta}$ and we consider the $\zeta\mbox{-admissible} $ $n$-tuple $(\sigma_1, \cdots, \sigma_n)$ defined by $\sigma_i:=\beta_i$\, for \,$i\ne 1, t\,,$ $\sigma_1:=\beta_1-\epsilon$\, and \,$\sigma_t:=\beta_t+\epsilon\,.$ 
From Proposition \ref{m-valutation}, we obtain \ $m(\sigma_1, \cdots, \sigma_n) -\delta_n= (\beta_1-\epsilon)^2-\beta_1^{\,2}+(\beta_t+\epsilon)^2-\beta_t^{\,2}=2\epsilon(\epsilon +\beta_t-\beta_1) >0\,,$ \,since \,$\beta_1\le \beta_t\,.$ Again, this contradicts the definition of $\delta_n$, so we conclude that (iii) is true.

iv) \ \ From (i), (ii), (iii), we deduce that, whatever the value of $\zeta \in \big[0,\lfloor \dfrac{n}{2} \rfloor\big]\cap\,\mathbb{Z}$, we necessarily have $-\pi<\beta_1\le \beta_2=\beta_3=\cdots=\beta_n\le \pi\,.$

Setting $ x_0:=\beta_2=\cdots=\beta_n\,,$ we have $\beta_1=2\pi\zeta -(n-1)x_0\,,$ where $\zeta \in \big[0,\lfloor \dfrac{n}{2} \rfloor\big]\cap\,\mathbb{Z}\ ,$ $-\pi < 2\pi\zeta -(n-1)x_0\, \le \, x_0 \, \le \, \pi\,;$  these  inequalities imply \ $ \dfrac{2 \zeta \pi}{n}\le x_0 < \dfrac{(2\zeta+1)\pi}{n-1}\ $  if \ $0 \le \zeta \le \lfloor \dfrac{n}{2} \rfloor -1\,$ \ and \ $ \dfrac{2 \zeta \pi}{n}\le x_0 \le \pi\,$ \ if $\zeta =\lfloor \dfrac{n}{2} \rfloor\,.$ \

In other words, necessarily the pair $(x_0, \zeta)$ belongs to the set  
$\Lambda$ defined by  

\vspace{.2cm}

$\Lambda:=\big\{(x, \vartheta): \  x \in [\dfrac{2 \vartheta \pi}{n}, \dfrac{(2\vartheta+1)\pi}{n-1})\,, \,\vartheta \in \big[0,\lfloor \dfrac{n}{2} \rfloor -1 \big]\cap\,\mathbb{Z}\, \ \mbox{\,\  \ or\ \ \,} x \in [\dfrac{2 \vartheta \pi}{n}, \pi]\,, \,\vartheta =\lfloor \dfrac{n}{2} \rfloor\big\}\,.$ 

\vspace{.2cm}

Conversely, it is easy to check that, for any $(x, \vartheta) \in \Lambda$,  the $n$-tuple\, $\big(2\pi\vartheta -(n-1)x, \, \underbrace{x, \cdots, x}_{n-1}\big)$ is $\vartheta\mbox{-admissible} $, with $\vartheta \in \big[0,\lfloor \dfrac{n}{2} \rfloor\big]\cap\,\mathbb{Z}$. Hence, if we define \\$F_{\vartheta}(x):= m\big(2\pi\vartheta -(n-1)x, \, \underbrace{x,\, \cdots,\, x}_{n-1}\big)\,,$  we have $\delta_n=F_{\zeta}(x_0)\ge F_{\vartheta}(x)\,, \mbox{\,for any} \,(x, \vartheta)  \in \Lambda\,.$ From Proposition \ref{m-valutation}, by means of easy calculations we obtain :

$F_{\vartheta}(x)=n(n-1)x^2-4\pi \vartheta n x+4\pi^2\vartheta(\vartheta+1)\,, \ \ \mbox{for every\ } \, (x, \vartheta) \in \Lambda\,.$ 

For any fixed $\vartheta \in \big[0,\lfloor \dfrac{n}{2} \rfloor-1\big]\cap\,\mathbb{Z}\,,\,$ the quadratic function $F_{\vartheta}$ reaches its absolute minimum at the point $\overline{x}_{\vartheta}=\dfrac{2\vartheta \pi}{n-1} \in [\dfrac{2 \vartheta \pi}{n}, \dfrac{(2\vartheta+1)\pi}{n-1})\,,\, \mbox{\ and\ we\ have}\,$
 
$ \overline{x}_{\vartheta}- \dfrac{2 \vartheta \pi}{n}=\dfrac{2\vartheta\pi}{n(n-1)} < \dfrac{\pi}{n-1}=\dfrac{(2\vartheta+1)\pi}{n-1}-\overline{x}_{\vartheta}\,;$ hence, for every $x \in [\dfrac{2 \vartheta \pi}{n}, \dfrac{(2\vartheta+1)\pi}{n-1})\,,$ there exists $\widehat{x}\in [\dfrac{2 \vartheta \pi}{n}, \dfrac{(2\vartheta+1)\pi}{n-1})\,$\, such that $F_{\vartheta}(\widehat{x}) > F_{\vartheta}(x)\,;$\, this implies that $\vartheta \ne \zeta\,,$\,for every $\vartheta \in \big[0,\lfloor \dfrac{n}{2} \rfloor-1\big]\cap\,\mathbb{Z}\,,\,$ and so we conclude that it must necessarily be $\zeta=\lfloor \dfrac{n}{2} \rfloor$\,. The function $F_{\zeta}(x)\,$ \, (with $\zeta=\lfloor \dfrac{n}{2} \rfloor$)\, is strictly decreasing on the interval $\big[\dfrac{2 \lfloor \dfrac{n}{2} \rfloor}{n} \pi, \pi\big]\,,$ so it has its unique maximum at the point $\dfrac{2 \lfloor \dfrac{n}{2} \rfloor \pi}{n}\,.$ \ We conclude that it must necessarily be $x_0=\dfrac{(n-1)\pi}{n}$ \ when $n$ is odd\,, and \ $x_0=\pi$ \  when $n$ is even.\\ So, with an easy calculation, we obtain that we necessarily have $\beta_1=\beta_2=\cdots=\beta_n=\pi\,$\, when $n$ is even, and $\beta_1=\beta_2=\cdots=\beta_n=\dfrac{(n-1)\pi}{n}\,$\, when $n$ is odd.

v) \ \ Since all matrices of $SU_n$ are diagonalizable, any $n$-tuple $(\,\underbrace{\lambda, \cdots, \lambda}_{n}\,)$ (with $\lambda \in [0, \pi]$ \,and \,$\dfrac{\lambda n}{2 \pi} \in \big[0 , \lfloor \dfrac{n}{2} \rfloor\,\big]\cap \mathbb{Z}\,$)  corresponds only to the matrix $e^{\lambda\textbf{i}} I_{_n} \in SU_n$. Therefore, taking into account Remarks-Definitions \ref{def-m(Q} (c), from (iv) we get that we have $m(Q)=\delta_n$ if and only if $Q=-I_{_n}$ when $n$ is even, and $Q=e^{\frac{(n-1)\pi\textbf{i}}{n}} I_{_n} \mbox{\ \ or\ \ } Q=e^{\frac{-(n-1)\pi\textbf{i}}{n}} I_{_n}  \mbox{\ when\ $n$\ is\ odd}\,.$  Hence, by means of Proposition \ref{m-valutation}, \  it is easy to check that \ $\delta_n=n\pi^2$ \,when $n$ is even, and \,$\delta_n=(n-\dfrac{1}{n})\pi^2$ \, when $n$ is odd; \,so the proof is complete.
\end{proof}

\section{Some geometrical properties of the Riemannian manifold $(SU_n, \phi)$}

\begin{thm}\label{distanza}
Let $d$ be the distance induced on $SU_n$ by the Frobenius metric $\phi$ and let $P, Q\in SU_n\,.$ With the same notation of Remarks \ref{prelim-Q-q} (c), without loss of generality
we can assume  $\zeta(P^*Q)\ge \zeta(Q^*P)$ $($so that $\zeta(P^*Q)\ge0$ by Remarks \ref{prelim-Q-q} $($c\,$))$; we denote by $\mu_1, \cdots , \mu_n$ the $n$ eigenvalues of $P^*Q$\, ordered so that $\ \arg(\mu_1) \le \arg(\mu_2) \le \cdots \le \arg(\mu_{n})\,$, and we set \ \ 
$\zeta:=\zeta(P^*Q)= \dfrac{1}{2 \pi} \sum\limits_{j=1}^n \arg(\mu_j) \ge 0\,.$ \ \ Then the following statements hold:

a) \ \ $d(P;Q)=$
$\left\lbrace
\begin{array}{cc}
 \sqrt{\sum\limits_{j=1}^{n} \big(\!\arg(\mu_j)\big)^2}\, \ \ \ \ \ \ \ \ \ \ \ \ \ \ \ \ \ \ \ \ \ \ \ \ \ \ \ \ \ \ \ \ \ \ \ \ \ \ \ \ \ \mbox{if\ $\zeta=0$}\,,\\
\!\!\sqrt{\sum\limits_{j=1}^{n-\zeta} \big(\!\arg(\mu_j)\big)^2 + \sum\limits_{j=n-\zeta+1}^{n} \big(2\pi-\arg(\mu_j)\big)^2}\,\ \ \ \ \ \ \ \ \mbox{if\ $\zeta\ge1$}
\end{array}\ \ ;
\right. $  

\smallskip

b) \ \ there exists a unique minimizing geodesic segment of \ $(SU_n, \phi)$ with endpoints $P$ and $Q$ \ if and only if  \ \ \ \ \ \ 
either  $\zeta=0$ 
\ \ \ \ \ \ or \ \ \ \ \ \  $\zeta\ge 1$ and $\mu_{n-\zeta} \ne \mu_{n-\zeta+1}\,;$ 
\smallskip

c) \  if $\zeta\ge 1$ \ and \ $\mu_{n-\zeta}=\mu_{n-\zeta+1}\,,$ \ \, 

after setting \ \ \ 
$\Theta(P^*Q)= \big\{X\in \mathfrak{su}_n : \exp(X)=P^*Q\,, \  \Vert X \Vert_{_\phi}= d(P,Q)\}\,,$

 $\nu_{_1}:=max\big\{j\in \{1,\cdots, n-\zeta\}: \mu_{n-\zeta+1-j}=\mu_{n-\zeta}\big\}\,$,

$\nu_{_2}:=max\big\{j\in \{1,\cdots, \zeta\}: \mu_{n-\zeta+j}=\mu_{n-\zeta+1}\big\}\,,$ 

the map: $X \mapsto \gamma(t):= P \exp(t X)\ \ (0 \leq t \leq 1)$ \ is a bijection from $\Theta(P^*Q)$ onto the set of minimizing geodesic segments of \ $(SU_n, \phi)$ with endpoints $P$ and $Q$, and  $\Theta(P^*Q)$ is a compact submanifold of $\mathfrak{su}_n$ diffeomorphic to the complex Grassmannian ${\bf Gr}(\nu_{_2}; \mathbb{C}^{(\nu_{_1}+\nu_{_2})})\,.$
\end{thm}

\begin{proof}
Since $\phi$ is bi-invariant, we have $\sqrt{m(P^*Q)}=d(I_{_n},P^*Q)=d(P,Q)\,;$  so part (a) follows from Proposition \ref{m-valutation}, while parts (b) and (c) follow from
 Propositions \ref{distanza-punti} and \ref{Theta(Q)}.
\end{proof}

\begin{thm}\label{diametral-points}
a) \ The diameter of $(SU_n, \phi)$ is

$\delta(SU_n, \phi)=$
$\left\lbrace
\begin{array}{cc}
\ \ \ \ \sqrt{n}\,\pi \ \ \ \ \ \ \ \ \ \ \ \ \ \ \ \ \mbox{if\ $n$\ is\ even\,}\\
\sqrt{n-\dfrac{1}{n}} \ \pi \ \ \ \ \ \ \ \ \ \ \ \ \mbox{if\ $n$\ is\ odd}
\end{array}\ \ .
\right. $

\smallskip

b) \ If $n$ is even  and $P$ is any matrix of $SU_n$,  then\  $-P$ is the unique diametral point  of  $P$ in $(SU_n,\phi)\,$  and the set of minimizing geodesic segments joining $P$ and $-P$ can be parametrized by the complex Grassmannian ${\bf Gr}(\frac{n}{2}; \mathbb{C}^{n})\,.$ 

c) \ If $n$ is odd and $P$ is any matrix of $SU_n$, then\ $P$ has precisely two diametral points $P^+$, $P^-$ in $(SU_n,\phi)\,,$ with\ \ $P^+=e^{\frac{(n-1)\pi\bf {i}}{n}} P$ \ and \ $P^-=e^{\frac{-(n-1)\pi\bf {i}}{n}} P\,,$ and the sets of minimizing geodesic segments joining $P$ with $P^+$ and $P$ with $P^-$  can both be parametrized by the complex Grassmannian ${\bf Gr}(\frac{n-1}{2}; \mathbb{C}^{n})\,.$ 
\end{thm}

\begin{proof}
The Theorem follows easily from Proposition \ref{diameter} and  Theorem \ref{distanza}, taking into account that, in case of diametral pairs, the values of the constants in part (c) of Theorem \ref{distanza} \ are : 
$\zeta=\nu_{_2}=\lfloor \dfrac{n}{2} \rfloor\,, \ \ \nu_{_1}=n-\lfloor \dfrac{n}{2} \rfloor\,.$ 
\end{proof}

\end{document}